\documentclass[11pt,a4paper]{amsart}
\usepackage{color}
\usepackage[T1]{fontenc}
\usepackage[utf8]{inputenc}
\usepackage{hyperref,amssymb,url,upref,verbatim,xspace,mathrsfs}
\usepackage{amsthm,amssymb,amsmath,latexsym,bm,stmaryrd}
\usepackage{spalign}
\RequirePackage[all,ps,cmtip]{xy}
\newdir^{ (}{{}*!/-5pt/\dir^{(}}
\newdir_{ (}{{}*!/-5pt/\dir_{(}}
\RequirePackage[mathscr]{eucal}
\usepackage{mathrsfs}\let\mathcal\mathscr
\theoremstyle{plain}

\newtheorem*{theorem*}{Main Theorem}
\newtheorem{theorem}{Theorem}
\newtheorem{proposition}[theorem]{Proposition}
\newtheorem{lemma}[theorem]{Lemma}
\newtheorem{corollary}[theorem]{Corollary}

\theoremstyle{definition}
\newtheorem{definition}[theorem]{Definition}
\newtheorem{remark}[theorem]{Remark}

\newtheorem{notation}[theorem]{Notation}

\AtBeginDocument{%
}

\newcounter{toto}
\def\thetoto{\arabic{toto}}
\let\oldmarginpar\marginpar
\def\marginpar#1{\refstepcounter{toto}\textsuperscript{\textup{[\thetoto]}}\oldmarginpar{\footnotesize\textsuperscript{[\thetoto]}\,#1}}

\def\mainmatter{\renewcommand{\baselinestretch}{1.1}\normalfont}

\makeatletter
\@namedef{subjclassname@2010}{\textup{2010} Mathematics Subject Classification}
\def\l@section{\@tocline{1}{0pt}{0pc}{}{}}
\def\l@subsection{\@tocline{2}{0pt}{1.5pc}{}{}}
\def\l@subsubsection{\@tocline{3}{0pt}{2pc}{}{}}
\makeatother
\makeatletter
\@namedef{subjclassname@2020}{%
   \textup{2020} Mathematics Subject Classification}
\makeatother

\newcommand{\C}{\mathbb{C}}\let\CC\C

\newcommand{\R}{\mathbb{R}}

\newcommand{\Z}{\mathbb{Z}}

\newcommand{\bD}{\boldsymbol{D}}
\newcommand{\shhom}{\mathcal{H}\!\mathit{om}}
\DeclareMathOperator{\conj}{c}
\DeclareMathOperator{\Conj}{C}
\let\Rhom\rh
\DeclareMathOperator{\tho}{\mathit{T}\shhom}

\DeclareMathOperator{\RH}{RH}

\newcommand{\rb}{\mathrm{b}}

\newcommand{\coh}{\mathrm{coh}}
\newcommand{\hol}{\mathrm{hol}}
\newcommand{\rhol}{\mathrm{rhol}}

\newcommand{\Mod}{\mathrm{Mod}}

\newcommand{\op}{\mathrm{op}}

\newcommand{\cc}{{\C\textup{-c}}}
\newcommand{\Rc}{{\R\textup{-c}}}

\newcommand{\XS}{X\times S}
\newcommand{\XbS}{\ov X\times S}

\newcommand{\YS}{Y\times S}

\newcommand{\DXS}{\shd_{\XS/S}}
\newcommand{\DXbS}{\shd_{\XbS/S}}

\DeclareMathOperator{\Char}{Char}

\DeclareMathOperator{\pD}{{}^\mathrm{p}\mathsf{D}}
\DeclareMathOperator{\rD}{\mathsf{D}}

\DeclareMathOperator{\DR}{DR}

\DeclareMathOperator{\Db}{\mathscr{D}b}
\DeclareMathOperator{\pDR}{{}^\mathrm{p}DR}

\DeclareMathOperator{\id}{Id}\let\Id\id

\DeclareMathOperator{\Sol}{Sol}
\DeclareMathOperator{\pSol}{{}^\mathrm{p}Sol}

\let\bar\overline
\let\hat\widehat

\let\ov\overline
\let\epsilon\varepsilon

\let\leq\leqslant
\let\geq\geqslant

\def\loccit{loc.\kern3pt cit.{}\xspace}
\def\cf{cf.\kern.3em}
\def\Cf{Cf.\kern.3em}
\def\eg{e.g.\kern.3em}

\def\resp{\text{resp.}\kern.3em}

\newcommand{\cbbullet}{{\raisebox{1pt}{$\sbullet$}}}
\newcommand{\sbullet}{{\scriptscriptstyle\bullet}}
\newcommand{\pOS}{p^{-1}\sho_S}

\def\shd{\mathcal{D}}

\let\cF F
\def\shg{\mathcal{G}}\let\cG G
\def\shh{\mathcal{H}}

\def\shj{\mathcal{J}}

\def\shm{\mathcal{M}}
\def\shn{\mathcal{N}}
\def\sho{\mathcal{O}}

\newcommand{\RedefinitSymbole}[1]{%
\expandafter\let\csname old\string#1\endcsname=#1
\let#1=\relax
\newcommand{#1}{\csname old\string#1\endcsname\,}%
}
\RedefinitSymbole{\forall} \RedefinitSymbole{\exists}

\def\to{\mathchoice{\longrightarrow}{\rightarrow}{\rightarrow}{\rightarrow}}

\def\hto{\mathrel{\lhook\joinrel\to}}

\def\To#1{\mathchoice{\xrightarrow{\textstyle\kern4pt#1\kern3pt}}{\stackrel{#1}{\longrightarrow}}{}{}}

\def\isom{\stackrel{\sim}{\longrightarrow}}

\let\oldbigoplus\bigoplus
\renewcommand{\bigoplus}{\mathop{\textstyle\oldbigoplus}\displaylimits}
\let\oldbigwedge\bigwedge
\renewcommand{\bigwedge}{\mathop{\textstyle\oldbigwedge}\displaylimits}
\let\oldbigcap\bigcap
\renewcommand{\bigcap}{\mathop{\textstyle\oldbigcap}\displaylimits}
\let\oldbigcup\bigcup
\renewcommand{\bigcup}{\mathop{\textstyle\oldbigcup}\displaylimits}

\begin{document}

\author[T. Monteiro Fernandes]{Teresa Monteiro Fernandes}
\address[T. Monteiro Fernandes]{Centro de Matem\'atica, Aplica\c{c}\~{o}es Funda\-men\-tais e Investiga\c c\~ao Operacional and Departamento de Matem\' atica da Faculdade de Ci\^en\-cias da Universidade de Lisboa, Bloco C6, Piso 2, Campo Grande, 1749-016, Lisboa
Portugal}
\email{mtfernandes@fc.ul.pt}
\thanks{The research of T. Monteiro Fernandes was supported by
Funda\c c{\~a}o para a Ci{\^e}ncia e Tecnologia, under the project: UIDB/04561/2020}
\title{Regularization of relative holonomic $\shd$-modules}

\date{\today}



\begin{abstract}
Let $X$ and $S$ be complex analytic manifolds where $S$ plays the role of a parameter space. Using the sheaf $\DXS^{\infty}$ of relative differential operators of infinite order, we construct functorially the regular holonomic $\DXS$-module $\shm_{reg}$ associated to a relative holonomic $\DXS$-module $\shm$,
extending to the relative case classical theorems by Kashiwara-Kawai: denoting by $\shm^{\infty}$ the tensor product of $\shm$ by $\DXS^{\infty}$ we explicit $\shm^{\infty}$ in terms of the sheaf of holomorphic solutions of $\shm$. As a consequence of the relative Riemann-Hilbert correspondence we conclude that $\shm^{\infty}$ and $\shm_{reg}^{\infty}$ are isomorphic.
\end{abstract}

\maketitle
\tableofcontents
\mainmatter

\vspace*{-2\baselineskip}\vskip0pt%
\section{Introduction}
The relative framework we deal with is associated to a projection $$p: X\times S\to S$$ where $X$ and $S$ are complex manifolds. 
Throughout this work we identify the relative cotangent bundle $T^*(X\times S/S)$ to $T^*X\times S$ and $d_X$ and $d_S$ will denote respectively the complex dimension of $X$ and of $S$.
Let $\DXS$ be the subsheaf of $\shd_{\XS}$ of operators commuting with $p^{-1}\sho_S$ and let $\Mod_{\coh}(\DXS)$ be the abelian category of coherent $\DXS$-modules.
A $\DXS$ -holonomic module is a coherent $\DXS$-module whose characteristic variety is contained in a product $\Lambda\times S$ where $\Lambda$ is $\C^*$-conic analytic lagrangian in $T^*X$ (\cf\cite{SS}, \cite{S}, \cite{MFCS1}). The datum of a strict (i.e, a $p^{-1}\sho_S$-flat) holonomic $\DXS$-module is equivalent to the datum of a flat family of holonomic $\shd_X$-modules with characteristic variety contained in $\Lambda$.

Let $\DXS^{\infty}$ denote the subsheaf of $\shd_{X\times S}^{\infty}$ of operators commuting with $p^{-1}\sho_S$.  As pointed out in \cite[Rem. 2, p. 406]{SKK}, the sheaf of rings $\DXS^{\infty}$ is faithfully flat over $\DXS$. Indeed the method of the proof of \cite[Th. 3.4.1]{SKK} which concerns the relative microdifferential case out of the zero section of $T^*(X\times S/S)$ adapts to the sheaves $\DXS$ and $\DXS^{\infty}$.

The relative setting means here that  $\shd_X$ and $\shd_X^{\infty}$ are replaced respectively by $\DXS$ and $\DXS^{\infty}$ and that we consider relative holonomic modules.
Our main main result is Theorem \ref{T1} which proves a relative version of the following Kashiwara-Kawai's Theorem  (\cite[Th. 1.4.9]{KK}): 
Let $\shd_X^{\infty}$ denote the sheaf of linear differential operators on $X$ with possibly infinite order. To any holonomic $\shd_X$-module one associates $\shm^{\infty}:=\shd_X^{\infty}\otimes_{\shd_X}\shm$ and, if $F=\Sol \shm$, then $\shm^{\infty}\simeq \Rhom(F, \sho_X)$.

The same authors introduce in loc.cit a regular holonomic $\shd_X$-module $\shm_{reg}$ contained in $\shm^{\infty}$  and prove in \cite[Theorem 5.2.1]{KK} a $\shd_X^{\infty}$-isomorphism: 
\begin{equation}\label{E130}
\shm^{\infty}\simeq \shm_{reg}^{\infty}
\end{equation} 

In (b) of Theorem \ref{T2} we extend this result to the relative setting.
The proof is based on the relative Riemann-Hilbert correspondence obtained in \cite{FMFS1} \& \cite{FMFS2} since one previous step is to prove that $(\cbbullet)_{reg}\simeq \RH^S(\Sol \cbbullet)[-d_X]$. The latter isomorphism is a contribution to the understanding of the functor $\RH^S$.

The task is not trivial although we dispose of a good notion of regularity recalled below, as well as of  the inspiration provided by the techniques in \cite{KK}. Let us explain why:

One big difference from the absolute to the relative case is that the triangulated category of $\DXS$-complexes having bounded holonomic cohomologies ($\rD^\rb_{\hol}(\DXS)$) is not stable under the inverse image functor by morphisms $f\times \Id:X'\times S\to X\times S$. 
Such constraint entails a loss of several functorial properties (for instance localization, algebraic supports cohomology).

The notions of $S$-$\R$- and $S$-$\C$-constructibility were introduced in \cite{MFCS1}  for objects in $\rD^\rb(\pOS)$ as well as a natural duality and a middle perversity $t$-structure on the triangulated category $\rD^\rb_{\cc}(p^{-1}\sho_S)$ whose objects have $S$-$\C$-constructible cohomologies. A perverse object with perverse dual is then equivalent to the datum of a flat family of perverse sheaves on $X$.

The lack of functorialities in $\rD^\rb_{\hol}(\DXS)$ prevents from stating an irregular relative Riemann-Hilbert correspondence by simply adapting the strategies used in the absolute case as treated by D'Agnolo-Kashiwara (\cf \cite{AK}). For a satisfactory functorial behaviour, regularity is necessary as proved in \cite{FMFS1}, \cite{FMFS2}.

Recall that a regular holonomic $\DXS$-module is a holonomic $\DXS$-module  satisfying the following condition: the (derived) holomorphic restriction to each fiber of $p$ is a regular holonomic complex on $X$. We also consider the associated triangulated category  ($\rD^\rb_{\rhol}(\DXS)$) of complexes having bounded regular holonomic cohomologies.

It is then natural to ask what kind of "regularity" can be associated to any holonomic $\DXS$-module.

Recall that the relative  Riemann-Hilbert equivalence was first proved in \cite{FMFS1} assuming that $d_S=1$:

The functor  $\pSol: \shm\mapsto\Rhom_{\DXS}(\shm, \sho_{X\times S})[d_X]$ from $\rD^\rb_{\rhol}(\DXS)$ to $\rD^\rb_{\cc}(p^{-1}\sho_S)$  admits a right and left  adjoint denoted by $\RH^S$ and thus $\pSol$ is an equivalence of categories.

In \cite{FMFS2}, the same authors proved that this equivalence holds true for arbitrary $d_S$. 

In the absolute case (meaning that S=pt) we recover Kashiwara's regular Riemann-Hilbert correspondence, and, if X=pt, we get the natural duality on the bounded derived category of complexes with $\sho_S$-coherent cohomologies. 

We now precise our results:

If $\shm$ is a holonomic $\DXS$-module, we define $$\shm^{\infty}:=\DXS^{\infty}\otimes_{\DXS}\shm$$  and we generalize this definition by flatness to $\rD^\rb_{\hol}(\DXS)$.

In our main result (Theorem \ref{T1}) we prove that if $\shm$ is an object of $\rD^\rb_{\hol}(\DXS)$  and $F=\pSol \shm$ then $\shm^{\infty}\simeq \Rhom_{p^{-1}\sho_S}(F, \sho_{
\XS})$ (to compare with \cite[Th. 1.4.9]{KK}). As a consequence one concludes in Theorem \ref{T2} that if $\shm$ is a holonomic $\DXS$-module then $\shm_{reg}\simeq\RH^S(\pSol \shm)$ and so $(\ref{E130})$ holds true in this setting.

The simplest example is the following: for a submanifold $Z$ of $X$, one has $$\RH^S(\C_{Z\times S}\otimes p^{-1}\sho_S)^{\infty}[-d_X]$$ $$\simeq \tho(\C_{Z\times S}, \sho_{X\times S})^{\infty}\simeq B^{\infty}_{Z\times S|X\times S}[-d]$$ $$\simeq \Rhom(\C_{Z\times S}, \sho_{X\times S})$$ where $d$ is the codimension of $Z$.

Another example is provided by \cite[page 814]{KK}, replacing $a\in\C$ by a holomorphic function $a(s)$ without zeros on some open $S:=\Omega\subset\C$.  For $X=\C$, we consider the $\DXS$-module (holonomic, non regular) defined by 
$$(x^2\partial_x-a(s))u(x, s)=0$$

We then obtain (\cf page 815 of \cite{KK}) an equivalent system substituing the generator $u$ by $u_0=u$ and introducing $u_1=-x\partial_xu$, 
\begin{equation}\label{E131}
  \spalignsys{
    x\partial_xu_0+u_1=0 ;
     -a(s)u_0-xu_1=0 
  }
\end{equation}
After multiplication by matrices in $\DXS^{\infty}$ (the matrices provided by \cite{KK} which now depend on the parameter $s$), one concludes a $\DXS^{\infty}$-isomorphism from the $\DXS^{\infty}$-module extension of $(\ref{E131})$ to the $\DXS^{\infty}$-module (with generators $w_0, w_1$) extension of the regular holonomic $\DXS$-module

\begin{equation}\label{E132}
  \spalignsys{
    x w_0-a(s)w_1=0 ;
    x\partial_x w_1=0 
  }
\end{equation}

We remark that \cite{KK} uses microlocal technics for the proof of the regularity of $\shm_{reg}$. With the more recent notion of microsupport (\cite{KS1}) and the results on \cite{SS}, the necessary tools in the relative framework (see Section Technical Lemmas) are easier to prove. Together with the relative Riemann-Hilbert correspondence, our task is much simplified, in particular we no longer need to microlocalize.

We warmly thank Pierre Schapira and Luisa Fiorot for attentively discussing several points.
We are deeply grateful for the referee's corrections which contributed to improve our work.

\section{A short reminder on the relative Riemann-Hilbert correspondence}

Below we summarize the background from \cite{MFCS1}, \cite{MFCS2}, \cite{FMFS1}, \cite{FMFS2} we shall need in the sequel.

\subsection{Holonomic and regular holonomic  $\DXS$-modules}
\begin{enumerate}
\item{We say that a $\pOS$-module is strict if it is flat over $\pOS$.}
\item{We recall that $\shm\in\Mod_{\coh}(\shd_{X\times S/S})$ is holonomic if the characteristic variety $\Char(\shm)$ is contained in $\Lambda\times S$, where $\Lambda$ is analytic $\C^*$-conic lagrangian subset of $T^*X$; we denote by $\rD^{\rb}_{\hol}(\DXS)$ the associated triangulated category whose objects are the bounded complexes with holonomic cohomologies.}

\item{There is a well defined duality functor $$\bD: \rD^{\rb}_{\hol}(\DXS)\to \rD^{\rb}_{\hol}(\DXS)^\op$$ given by
$$\bD \shm:=\Rhom_{\DXS}(\shm, \DXS\otimes _{\sho_{X\times S}}\Omega_{\XS/S}^{\otimes^{-1}})[d_X])$$ where $\Omega_{\XS/S}$ denotes the sheaf of relative differential forms of maximal degree.}
\item{$\bD$ is an involution, i.e. $\bD\bD=\Id$.}
\item{We recall a tool introduced in \cite{MFCS1}, the holomorphic restriction to each fiber of $p$:  $$\forall s\in S,\, Li^*_s(\cbbullet):=\cbbullet\overset{L}{\otimes}_{p^{-1}\sho_S}p^{-1}(\sho_S/\shj_s)$$ where $\shj_s$ is the maximal ideal of functions vanishing in $s$.}
\item{ A Nakayama's Lemma variation: Let $\shm\in\rD^\rb_\hol(\DXS)$ and assume that $L i^*_{s_o}\shm=0$ for each $s_o\in S$. Then $\shm=0$.}
\item{Let $\shm$ be an object of $\rD^{\rb}_{\hol}(\DXS)$. Then $\bD \shm$ is concentrated in degree zero and $\shh^0\bD \shm$ is strict if and only if $\shm$ is itself concentrated in degree zero and $\shh^0\shm$ is a strict $\DXS$-module.}
\item{We say that $\shm\in\Mod(\shd_{X\times S/S})$ is regular holonomic if it is holonomic and $\forall s\in S, Li^*_s\shm\in\rD^{\rb}_{\rhol}(\shd_X)$; we denote by $\rD^{\rb}_{\rhol}(\DXS)$ the associated triangulated full subcategory of $\rD^\rb_{\hol}(\DXS)$.}
\item{$\rD^{\rb}_{\rhol}(\DXS)$ is stable by duality.}
\item{$\Mod_{\hol}(\DXS)$ and $\Mod_{\rhol}(\DXS)$ are closed under taking extensions in $\Mod(\DXS)$ and subquotients in $\Mod_{\coh}(\DXS)$.}
\end{enumerate}

\subsection{$S$-constructibility}

\vspace{2mm}

We say that a  sheaf $L$ of $p^{-1}\sho_S$-modules is $S$-locally constant coherent if, locally on $X\times S$, $L$ is isomorphic to $p^{-1}G$ where $G$ is an $\sho_S$-coherent module. Such an $L$ is also called $S$-local system. We recall the following full triangulated subcategories of $\rD^\rb(p^{-1}\sho_S)$.

\begin{itemize}
\item{An object $F\in\rD^\rb(p^{-1}\sho_S)$ is an object of $\rD^\rb_\cc(\pOS)$ if there exists a $\C$-analytic stratification $(X_{\alpha})_{\alpha\in A}$ of $X$, such that $\forall j\in\Z, \forall \alpha\in A, \shh^jF|_{X_{\alpha}\times S}$ is $S$-locally constant coherent. We say for short that $F$ is $S$-$\C$-constructible.} 

\item{Replacing $\C$-analyticity by subanalyticity with respect to the real analytic manifold $X_{\R}$ undelying $X$, we obtain the notion of $S-\R$-constructibility and the corresponding triangulated category $\rD^\rb_{\Rc}(p^{-1}\sho_S)$. $\rD^\rb_\cc(\pOS)$ is a full subcategory of $\rD^\rb_{\Rc}(p^{-1}\sho_S)$.}
\item{If $F\in\rD^\rb_\Rc(\pOS)$ then for each $x\in X$, $F|_{\{x\}\times S}$ belongs to $\rD^\rb_{\coh}(\sho_S)$.}
\item{There is a natural duality functor $\bD: \rD^\rb_\Rc(\pOS)\to \rD^\rb_\Rc(\pOS)^\op$ which is an involution given by
$$\bD F=\Rhom_{\pOS}(F,\pOS)[2 d_X]$$}

\item{$\rD^\rb_\cc(\pOS)$ is stable by duality.}

\end{itemize}

\subsection{ A middle perversity t-structure on $\rD^\rb_\cc(\pOS)$.}

We consider the two subcategories $\pD^{\leq0}_\cc(\pOS)$ and $\pD^{\geq0}_\cc(\pOS)$ of $\rD^\rb_\cc(\pOS)$ defined as follows:

$F\in \pD^{\leq0}_\cc(\pOS)$ (\resp $F\in \pD^{\geq0}_\cc(\pOS)$) if 
for an adapted $\mu$-stratification $(X_\alpha)_{\alpha\in A}$, noting \hbox{$i_\alpha:X_\alpha\hto X$} 
\begin{align*}
\forall\alpha\text{ and }\forall j>-\dim(X_\alpha),\quad\shh^j(i^{-1}_\alpha F)&=0 \\
\text{(resp.)}\,\forall \alpha\text{ and }\forall j<-\dim(X_\alpha),\quad\shh^j(i^{!}_\alpha F)&=0.
\end{align*}

We say that $F$ of $\rD^\rb_\cc(\pOS)$ is \emph{perverse} if $F \in\pD^{\leq0}_\cc(\pOS)$ and $F\in\pD^{\geq0}_\cc(\pOS)$, that is $F$ belongs to the heart of the $t$-structure defined above.

\remark Note that $\bD$ is not $t$-exact for this $t$-structure, in particular it does not preserve perversity.

\begin{theorem}[\cite{MFCS2}] For a given object $F\in \rD^\rb_\cc(\pOS)$,
$F$ and $\bD F$ are perverse if and only if $\forall s_o\in S, Li^*_{s_o}(F)$ is perverse in $\rD^\rb_{\cc}(\C_X)$.
\end{theorem}

 \subsection{Link with holonomicity}
 We have the following link with holonomic $\DXS$-modules. Let us note $\pSol\shm=\Rhom_{\DXS}(\shm, \sho_{\XS})[d_X]$ and $\pDR \shm=\Rhom_{\DXS}(\sho_{\XS}, \shm)[d_X]$.

 Then (\cf \cite{FMF}, \cite{MFCS1}, \cite{MFCS2}):
 \begin{itemize}
\item
{$\Sol, \DR:\rD^{\rb}_{\hol}(\DXS)$ take values in $\rD^{\rb}_{\cc}(\pOS)$ and $\bD\pSol=\pDR=\pSol\bD$. } 
\item{If $\shm\in \Mod_{\hol}(\DXS)$ then $\pDR\shm$ is perverse (\cf \cite[Theorem 4.1]{FMF}).}
\item{If $F$ is such that $\bD F$ is perverse then $\RH^S(F)$ is concentrated in degree zero (\cf \cite[Th. 4.1]{FMF}). In particular, for any holonomic $\DXS$-module, $\RH^S(\pSol \shm)$ is concentrated in degree zero.}
 \item{Given $\shm\in\rD^\rb_{\hol}(\DXS)$, $\shm$ and $\bD\shm$ are strict $\DXS$-modules if and only if $\pSol \shm$ and $\pDR \shm=\bD\pSol\shm$ are perverse. }
 \end{itemize}

\subsection{The functor $\RH^S$}With the subanalytic tools developed in \cite{MFP1}, \cite{MFP2}, the functor $\RH^S$ was first introduced in \cite{MFCS2}, followed by \cite{FMFS1} (case $d_S=1$) and by  \cite{FMFS2} (general case). Kashiwara's functor $\RH$ (\cf \cite{Ka3}) is recovered with $d_S=0$. 
 Below we give a short reminder of its construction and main results:

 Let $\rho_S: X\times S\to X_{sa}\times S$ be the natural morphism of sites introduced in \cite{MFP2}. The functor $\rho_S^{-1}$ admits a left adjoint $\rho_{S!}$ which is exact. We note $\sho_{X\times S}^{t, S}$ the relative subanalytic sheaf on $X_{sa}\times S$ associated in \cite{MFP2} to the subanalytic sheaf $\sho^t_{X\times S}$ on $(X\times S)_{sa}$ (introduced in \cite{KS6}, see also \cite{LP08}).
 
 The functor $\RH^S$ on $\rD^b(\pOS)^\op$ is given by
 
 $$\RH^S(\cbbullet):=\rho_S^{-1}\Rhom_{\rho_{S*}\pOS}(R\rho_{S*}(\cbbullet), \sho^{t,S}_{X\times S})[d_X]$$

\begin{theorem}[\cite{MFCS2}, \cite{FMFS1}, \cite{FMFS2}]\label{T20}

\begin{enumerate}
\item{$\RH^ S$ induces an equivalence of categories: $\rD^\rb_{\cc}(p^{-1}\sho_S)^\op\to \rD^\rb_{\rhol}(\DXS)$ compatible with duality.}

\item{ $F$ is perverse with a perverse dual  if and only if $\RH^S(F)$ is strict and concentrated in degree zero.}

\item{\,For $F\in \rD^\rb_{\cc}(p^{-1}\sho_S)$ and $\shm\in\rD^\rb_{\rhol}(\DXS)$, we have a natural isomorphism in $\rD^\rb_{\cc}(p^{-1}\sho_S)$ $$\Rhom_{\DXS}(\shm, \RH^S(F)[-d_X])\overset{\sim}{\rightarrow} \Rhom_{\DXS}(\shm, \Rhom_{p^{-1}\sho_S}(F,\sho_{X\times S}))$$}
\end{enumerate}
\end{theorem}

\subsection { Topological aspects of $\sho_S$} $\sho_S$ is a sheaf of complete bornological algebras (multiplicatively convex sheaf of Fr\'echet algebras over $S$). In the category of sheaves of complete bornological modules over $\sho_S$ (denoted by $\mathrm{Born}(\sho_S))$, C. Houzel (\cf\cite{H}) introduced a notion of tensor product 
$\cbbullet\hat{\otimes}_{\sho_S}\cbbullet$. To the latter one associates a family of functors $\cbbullet\hat{\otimes}\shm$ on the category of bornological vector spaces, depending functorialy on $\shm\in \mathrm{Born}(\sho_S)$ (\cf \cite[Section 3.4]{SS}).
We have \begin{equation}\label{E122}\sho_{X\times S}|_{\{x\}\times S}\simeq \sho_{X,x}\hat{\otimes}\sho_S
\end{equation}

Then (\ref{E122}) shows that $\sho_{X\times S}|_{\{x\}\times S}$ is a so-called FN-free as well as a DFN-free $\sho_S$-module (\cf \cite[page 25]{SS} for the definition and also \cite{RR}). 

In particular, given another complex manifold $Y$, we have
\begin{equation}\label{E123}
\sho_{X\times Y\times S}|_{\{(x,y)\}\times S}\simeq (\sho_{X\times S}|_{\{x\}\times S})\hat{\otimes}_{\sho_S}(\sho_{Y\times S}|_{\{y\}\times S})
\end{equation}

\section{Technical Lemmas}

In order to prove the main theorem we will need the following results:

\subsection{Complements on $S$-$\R$-constructible sheaves}
We refer to \cite[Chapter VIII]{KS1} for the background on constructibility. 

\notation\label{NP} For short we shall keep the notations $p$ as well as $\pOS$ without referring to the manifold $X$ whenever there is no risk of ambiguity.

Let $X$ and $Y$ be complex manifolds. Let $q_1: X\times Y\times S\to X\times S$ be the first projection and $q_2: X\times Y\times S\to Y\times S$ be the second projection,  which is illustrated by the following commutative diagram below.
\begin{equation}\label{eq:commutpij}
\begin{array}{c}
\xymatrix{
X\times Y\times S\ar[r]^-{q_1}\ar[d]_{q_2}\ar[dr]^{p}&X\times S\ar[d]^p\\
Y\times S\ar[r]^-{p}&S
}
\end{array}
\end{equation}

\begin{lemma}\label{L2}  For any $F\in\rD^\rb_{\Rc}(p^{-1}\sho_S)$ on $X\times S$ and any object $\shg$ of $\rD^\rb(\pOS)$ on $Y\times S$ the functorial morphism
\begin{equation}\label{E128}T(F):=q_1^{-1}\Rhom_{p^{-1}\sho_S}(F, p^{-1}\sho_S)\otimes^L_{p^{-1}\sho_S}q_2^{-1}\shg
\end{equation} 
$${\longrightarrow}T'(F):=\Rhom_{p^{-1}\sho_S}(q_1^{-1}F, q_2^{-1}\shg).
$$ is an isomorphism.
\end{lemma}
\begin{proof} 
The proof is now simpler than that of Lemma B.3 of \cite{KK} since we dispose of the notion of microsupport and of its properties (\cf \cite[Chap.V]{KS1}). It is sufficient to check the isomorphism locally. Furthermore, arguing by induction on the length of $F$, we may assume that $F$ is in degree zero, that is, $F$ is an $S$-$\R$-constructible sheaf. 

We recall the following result (\cf Lemma.\,A.9 in the complete version of \cite{FMFS2}, https://arxiv.org/pdf/2203.05444.pdf):

\begin{lemma}\label{lem:epiconst}
Let $F$ be an $S$-$\R$-constructible sheaf on $\XS$. Then there exist
\begin{itemize}
\item
a locally finite covering $(U(\sigma))_{\sigma\in\Delta}$ of $X$ by open subanalytic relatively compact subsets of $X$,
\item
for each $\sigma\in\Delta$ a coherent $\sho_S$-module $G_\sigma(F)$ on $S$,
\item
and an epimorphism $\bigoplus_{\sigma\in\Delta}\CC_{U(\sigma)}\boxtimes G_\sigma(F)\to F$.
\end{itemize}
\end{lemma}

Let us assume for a moment that $F=\C_{U}\boxtimes G$ for some open relatively compact subanalytic subset $U$ of $X$ and for some coherent $\sho_S$-module $G$. In that case, the proof of Lemma \ref{L2} is as follows.
Regarding the left hand term of $(\ref{E128})$ we have a chain of isomorpisms:
$$q_1^{-1}\Rhom_{\pOS}(\C_U\boxtimes G, \pOS)\otimes^L_{\pOS}q_2^{-1}\shg$$$$\simeq 
q_1^{-1}\Rhom(\C_{U\times S}, \Rhom_{\pOS}(p^{-1} G, \pOS))\otimes^L_{\pOS}q_2^{-1}\shg$$$$\underset{(a)}{\simeq} q_1^{-1}\Rhom(\C_{U\times S}, \C_{X\times S})\otimes q_1^{-1}\Rhom_{\pOS}(p^{-1} G, \pOS)\otimes^L_{\pOS}q_2^{-1}\shg$$$$\underset{(b)}{\simeq}
q_1^{-1}\Rhom(\C_{U\times S}, \C_{X\times S})\otimes \Rhom_{\pOS}(p^{-1} G, q_2^{-1}\shg)$$
Isomorphism $(a)$ follows by  \cite[Prop. 5.4.14 (ii)]{KS1} and isomorphism $(b)$ follows by the coherence of $G$.

 Similarly, the right hand term of $(\ref{E128})$ becomes isomorphic to $$\Rhom(q_1^{-1}\C_{U\times S}, \Rhom_{\pOS}(p^{-1} G, q_2^{-1}\shg))$$$$
 \simeq \Rhom(\C_{U\times Y\times S}, \Rhom_{\pOS}(p^{-1} G, q_2^{-1}\shg))$$
 
We have $$q_1^{-1}\Rhom(\C_{U\times S}, \C_{X\times S})\simeq \Rhom(\C_{U\times Y\times S}, \C_{X\times Y\times S})$$ Thus, for $F=\C_U\boxtimes G$,  Lemma \ref{L2}  follows by  \cite[Prop. 5.4.14 (ii)]{KS1}. 

As a consequence, Lemma \ref{L2} holds true for sheaves of the form $(\ast)\, \bigoplus_{\sigma\in\Delta}\CC_{U(\sigma)}\boxtimes G_\sigma(F)$. 

We shall now prove the general case ($F\in \Mod_{\R-c}(\pOS)$) by a standard argument. 
The epimorphism of Lemma \ref{lem:epiconst} induces the following exact sequence

$$0\to F'\to K\to F\to 0$$ where $K$ has the form $(\ast)$, thus $K$ and $F'$ belong to $\Mod_{\R-c}(\pOS)$. 
  We consider the associated distinguished triangles
  $$T(F)\to T(K)\to T(F')\overset{+1}{\to}$$
  $$T'(F)\to T'(K)\to T'(F')\overset{+1}{\to}$$
  
Thus \eqref{E128} reads $T(K)\simeq T'(K)$ in $\rD^\rb_{\R-c}(\pOS)$.
There exist integers $N<M$ only depending on $T$, $T'$ and $\shg$ such that the j-cohomology groups of $T(\cbbullet), T'(\cbbullet)$, with $\cbbullet$ replaced by $F,F', K$ (see \eqref{E128}), vanish for $j\notin [N, M]$. We have $\shh^NT(K)\simeq \shh^NT'(K)$ thus $\shh^N T(F)\to\shh^N T'(F)$ is injective (since $\shh^{N-1} T(F')=0=\shh^{N-1} T'(F')$). As $F$ is arbitrary, the same holds true for $F$ replaced by $F'$. By the Five Lemma it follows that $\shh^N T(F)\simeq\shh^N T'(F)$ and so 
$\shh^N T(F')\simeq\shh^N T'(F')$ again because $F$ is arbitrary. We then pursue recursively this argument which ends after a finite number of steps.

\end{proof}

\subsection{A complement on relative holonomic modules}
Let $X$ and $Y$ be complex manifolds and let us 
consider diagram (\ref{eq:commutpij}).

\begin{lemma}\label{L4}
Let $\shm\in\rD^\rb_{\hol}(\DXS)$ Then we have a natural isomorphism
 $$q_1^{-1}\Rhom_{\DXS}(\shm, \sho_{X\times S})\otimes_{p^{-1}\sho_S}q_2^{-1}\sho_{Y\times S}$$ $$\overset{\sim}{\to}\Rhom_{q_1^{-1}\DXS}(q_1^{-1}\shm, \sho_{X\times Y\times S})$$
\end{lemma}
\begin{proof}
We adapt Proposition 1.4.3 of \cite{KK}. Since the morphism is well defined, it is enough to prove that, for any $x\in X,\, y\in Y$, it induces an isomorphism
 $$\Rhom_{\DXS}(\shm, \sho_{\XS})|_{\{x\}\times S}\otimes_{\sho_{S}}{\sho_{\YS}}|_{\{y\}\times S}$$ $$\simeq\Rhom_{q_1^{-1}\DXS}(q_1^{-1}\shm, \sho_{X\times Y\times S})|_{\{(x,y)\}\times S}$$ For any $s\in S$, in a neighbourhood of $(x,s)$, we now replace $\shm$ by a bounded locally free $\DXS$-resolution 
$(\DXS^{k}, p_k)_{k\in \Z}\underset{QIS}{\to} \shm$. Then we may assume that $\Rhom_{\DXS}(\shm, \sho_{\XS})|_{\{x\}\times S}$ is quasi-isomorphic to the complex $((\sho^k_{\XS})|_{\{x\}\times S}, p_k^\top)$ and that $\Rhom_{q_1^{-1}\DXS}(q_1^{-1}\shm, \sho_{X\times Y\times S})|_{\{(x,y\}\times S}$ is quasi-isomorphic to the complex $((\sho_{X\times Y\times S}^k)|_{\{(x,y)\}\times S}, p_k^\top)$.

 We have  $\sho^k_{\XS}|_{\{x\}\times S}\simeq \sho^k_{X,x}\hat{\otimes}\sho_S$ and $\sho_{X\times Y\times S}^k|_{\{(x,y)\}\times S}\simeq\sho_{X\times Y, (x,y)}\hat{\otimes} \sho_S$ and so $\sho^k_{\XS}|_{\{x\}\times S}$ as well as $\sho_{X\times Y\times S}^k|_{\{(x,y)\}\times S}$ are FN-free $\sho_S$-modules in the sense of \cite{RR}.

Since $\Rhom_{\DXS}(\shm, \sho_{X\times S})|_{\{x\}\times S}$ has $\sho_S$-coherent cohomologies 
we are in conditions of applying Proposition 3.13 of \cite{SS}, and, in view of (\ref{E123}), to conclude quasi isomorphisms $$(\sho^k_{X\times Y\times S}, p_k^\top)|_{\{(x,y)\}\times S}\simeq(\sho^k_{X\times S}, p_k^\top)|_{\{x\}\times S}\hat{\otimes}_{\sho_S}(\sho_{\YS})|_{\{y\}\times S}$$ $$\simeq\Rhom_{\DXS}(\shm, \sho_{\XS})|_{\{x\}\times S}\otimes_{\sho_{S}}{\sho_{\YS}}|_{\{y\}\times S}$$
\end{proof}

\section{Main result}

\subsection{Statement and proof of the main result}Let $\Delta$ denote the diagonal of $X\times X$. 

The canonical section of $i^{-1}_{\Delta\times S} \shh^{d_X}_{\Delta\times S}(\sho_{X\times X\times S})\otimes_{\sho_{X\times S}}\Omega_{X\times S/S}=i^{-1}_{\Delta\times S}B^{\infty}_{X\times S|X\times X\times S}\otimes_{\sho_{X\times S}}\Omega_{X\times S/S}$ corresponding to the global section $1$ of $\DXS^{\infty}$ allows to define an isomorphism of sheaves of rings  
\begin{equation}\label{E129}
\DXS^{\infty}\simeq  i^{-1}_{\Delta\times S}B^{\infty}_{X\times S|X\times X\times S}\otimes_{\sho_{X\times S}}\Omega_{X\times S/S}
\end{equation}
$$\simeq i^{-1}_{\Delta\times S}R\Gamma_{\Delta\times S}(\sho_{X\times X\times S})\otimes_{\sho_{X\times S}}\Omega_{\XS/S}[d_X]$$

\begin{theorem}\label{T1} 

Let $\shm\in\rD^\rb_{\hol}(\DXS)$. 

Let $F=\Sol \shm=\Rhom_{\DXS}(\shm, \sho_{X\times S})$.
Then we have a natural isomorphism in $\rD^\rb(\DXS^{\infty})$
$$\shm^{\infty}\simeq \Rhom_{p^{-1}\sho_S}(F, \sho_{X\times S})$$
\end{theorem}

\begin{proof}

In view of $(\ref{E129})$, we have isomorphisms 

$\DXS^{\infty}\otimes_{\DXS}\shm$
\begin{align*}
\simeq i^{-1}_{\Delta\times S}R\Gamma_{\Delta\times S}(\sho_{X\times X\times S})\otimes_{\sho_{X\times S}}\Omega_{\XS/S}[d_X]\otimes_{\DXS}\bD\bD\shm\\ 
\simeq i^{-1}_{\Delta\times S}R\Gamma_{\Delta\times S}(\Rhom_{q_1^{-1}\DXS}(q_1^{-1}\bD\shm, \sho_{X\times X\times S}))[2d_X]
\end{align*}

According to Lemma \ref{L4}, we have $$\Rhom_{q_1^{-1}\DXS}(q_1^{-1}\bD\shm, \sho_{X\times X\times S})$$ $$\simeq q_1^{-1}\Rhom_{\DXS}(\bD \shm, \sho_{\XS})\otimes_{\pOS}q_2^{-1}\sho_{\XS}$$

On the other hand the holonomicity of $\shm$ implies the isomorphism $$\Rhom_{\DXS}(\bD \shm, \sho_{\XS})$$ $$\simeq \Rhom_{p^{-1}\sho_S}(\Rhom_{\DXS}(\shm, \sho_{\XS}), p^{-1}\sho_S) $$ According to Lemma \ref{L2} with $X=Y$, $F=\Rhom_{\DXS}(\shm, \sho_{\XS})$ and $\shg=\sho_{\XS}$, we conclude a natural isomorphism $$\Rhom_{q_1^{-1}\DXS}(q_1^{-1}\bD\shm, \sho_{X\times X\times S})$$ $$\simeq \Rhom_{\pOS}(q_1^{-1}\Rhom_{\DXS}(\shm, \sho_{\XS}), q_2^{-1}\sho_{\XS})$$

Applying $i^{-1}_{\Delta\times S}R\Gamma_{\Delta\times S}$ and the shift $[2d_X]$ to both terms we finally deduce a natural isomorphim $$\shm^{\infty}\simeq \Rhom_{p^{-1}\sho_S}(\Rhom_{\DXS}(\shm,\sho_{X\times S}), \sho_{\XS})$$ which follows by the sequence of isomorphisms 
\begin{equation}\label{E20}
\end{equation}
$$i^{-1}_{\Delta\times S}R\Gamma_{\Delta\times S}(q_2^{-1}\sho_{X\times S})$$
$$\simeq i^{-1}_{\Delta\times S}\Rhom(\C_{\Delta\times S}, q_2^{-1}\sho_{X\times S})$$
$$\underset{(a')}{\simeq} i^{-1}_{\Delta\times S}(\Rhom(\C_{\Delta\times S}, \C_{X\times X\times S})\otimes q_2^{-1}\sho_{X\times S})$$
$$\underset{(b')}\simeq\sho_{X\times S}[-2d_X]$$

where $(a')$ follows by \cite[Prop. 5.4.14 (ii)]{KS1} and $(b')$ by the commutation of the functors $\otimes$ and $i^{-1}_{\Delta\times S}$.
\end{proof}
\begin{corollary}\label{C1} 
\begin{enumerate}
\item{We have an isomorphism of functors on $\rD^\rb_{\cc}(\pOS)$: $$\RH^S(\cbbullet)^{\infty}[-d_X]\simeq\Rhom_{\pOS}(\cbbullet, \sho_{X\times S})$$}
\item{Let $\shm, \shn\in\rD^\rb_{\rhol}(\DXS)$. We have a natural isomorphism in $\rD^\rb_{\cc}(p^{-1}\sho_S)$ 
$$\Rhom_{\DXS}(\shm,\shn)\simeq \Rhom_{\DXS}(\shm, \shn^{\infty})$$}
\end{enumerate}
\end{corollary}

\begin{proof} (a) is an immediate consequence of Theorem \ref{T1} since $\Sol[d_X] \circ \RH^S=\Id$. 

(b) We set $\shn=\RH^S(G)$ with $G=\pSol(\shn)\in \rD^\rb_{\cc}(\pOS)$, then $\shn^{\infty}\simeq \Rhom_{p^{-1}\sho_S}(G,\sho_{X\times S})[d_X]$ and the result follows by Theorem \ref{T20} (c).

\end{proof}

\begin{definition}\label{D1}If $\shm$ is a holonomic $\DXS$-module, we denote by $\shm_{reg}$ the subsheaf of $\shm^{\infty}$ of local sections $u$ satisfying the following condition: there exists a coherent ideal $\shj$ in $\DXS$ such that $\shj u=0$ and $\DXS/\shj$ is regular holonomic.
\end{definition}
\begin{lemma}\label{L3} $\shm_{reg}$ is a $\DXS$-module .
\end{lemma}
\begin{proof}
The proof is similar to that in Prop.1.1.20 in \cite{KK}. If $u$ is a local section of $\shm_{reg}$, let $\shj$ be a left ideal of $\DXS$ as in Definition \ref{D1}, and let $P\in\DXS$; then the left ideal $\shj'$ of operators $Q$ such that $QP\in \shj$ is coherent and $\DXS/\shj'$ is isomorphic to a coherent $\DXS$-submodule of $\DXS/\shj$ hence, in view of (j) of $2.a$, it is regular holonomic so that conditions on Definition \ref{D1} are satisfied by $Pu$.
\end{proof}

Clearly the correspondence $$\shm\in\Mod_{\hol}(\DXS)\mapsto \shm_{reg}\in\Mod_{\rhol}(\DXS)$$ defines a left exact functor.

\begin{theorem}\label{T2} 
\begin{enumerate}
\item{Let $\shn\in \Mod_{\rhol}(\DXS)$. Then $\shn=\shn_{reg}$.}
\item{Let $\shn\in \Mod_{\hol}(\DXS)$. Then $\shn_{reg}$ is a regular holonomic $\DXS$-module isomorphic to $\RH^S(\pSol \shn)$. In particular $\shn^{\infty}\simeq \shn_{reg}^{\infty}$.}
\end{enumerate}
\end{theorem}

\begin{proof} (a)\,\,By the assumption of regularity, we derive a natural inclusion $\shn\subset\shn_{reg}$. Let us now prove the inclusion $\shn_{reg}\subset \shn$. 
Let $u$ be a local section of $\shn_{reg}$ and let $\shj$ be a left coherent ideal of $\DXS$ such that $\shj u=0$ and such that $\DXS/\shj$ is regular holonomic. We thus deduce a natural morphism $\phi:\DXS/\shj\to \shn^{\infty}$ as the composition of $\DXS/\shj\twoheadrightarrow\DXS u\hookrightarrow \shn^{\infty}$. Applying Corollary \ref{C1}(b) to the cohomologies of degree zero with $\shm=\DXS/\shj$, $\phi$ factors through $\shn$ thus $\DXS u\subset \shn$.

(b) According to Theorem \ref{T1} we have a $\DXS$-linear isomorphim $$\Phi: \shn^{\infty}\simeq \RH^S(\pSol \shn)^{\infty}$$ In view of $(a)$ we conclude by a similar argument that $\Phi(\shn_{reg})$ is contained in $\RH^S(\pSol \shn)$. Similarly, using $\Phi^{-1}$, we conclude that $\shn_{reg}$ contains $\Phi^{-1}(\RH^S(\pSol \shn))$. Thus $\Phi$ provides the desired $\DXS$-isomorphism.

\end{proof}
\subsection{Example}

We will assume that $d_S=1$. Our goal is to explicit $(\cbbullet)^{\infty}$ in the case of the relative hermitian duality(\cf \cite{MFCS3}) by proving the relative variant of Remark 2.1 in \cite{Ka4}.

We denote by $\Db_{\XS}$ the sheaf of distributions on the really analytic manifold $X_{\R}\times S_{\R}$ underlying $X\times S$ and by  $\Db_{\XS/S}$ the subsheaf of $\Db_{\XS}$ of germs of distributions holomorphic along $S$. We call $\Db_{\XS/S}$ the sheaf of relative distributions. We denote by $\bar{X}$ the complex conjugate manifold of the manifold $X$.
We recall the main results (Theorem 2) in \cite{MFCS3}:
 
\begin{enumerate}
\item
{The relative Hermitian duality functor
\[
\Conj_{X,\ov X}^S
(\cbbullet):=\Rhom_{\DXS}(\cbbullet, \Db_{\XS/S})
\]
induces an equivalence
$$\Conj_{X,\ov X}^S:\rD^\rb_{\rhol}(\DXS)\isom\rD^\rb_{\rhol}(\DXbS)^\op$$}

\item{$$
\Conj_{\ov X,X}^S\circ\Conj_{X,\ov X}^S\simeq\id.
$$}
\item{Moreover, the relative conjugation functor $$\conj_{X,\ov X}^S:=\Conj_{X,\ov X}^S\circ \bD$$ induces an equivalence
$$
\conj_{X,\ov X}^S:\rD^\rb_{\rhol}(\DXS)\isom\rD^\rb_{\rhol}(\DXbS),
$$
and there is an isomorphism of functors
$$\pSol_{\ov X}\circ\conj_{X,\ov X}^S\simeq\pSol_X:\rD^\rb_{\rhol}(\DXS)\to\rD^\rb_{\cc}(\pOS).
$$
}

\end{enumerate}

Let $B_{X_{\R}\times S_{\R}}$ be the sheaf of Sato's hyperfunctions on $X_{\R}\times S_{\R}$. Let $B_{X\times S/S}$ denote the subsheaf of $B_{X_{\R}\times S_{\R}}$ of germs of hyperfunctions which are holomorphic along the parameter manifold $S$.

\begin{proposition}\label{P2}
Let $\shm$ be an object of $\rD^\rb_{\rhol}(\DXS)$. Then we have $$\Conj_{X,\ov X}^S(\shm)^{\infty}\simeq \Rhom_{\DXS}(\shm, B_{X\times S/S})$$ 

\end{proposition}
\begin{proof}
We note that, by definition of $B_{X_{\R}\times S_{\R}}$, we have an isomorphism of $\DXS^{\infty}$-modules 
$$B_{X\times S/S}\simeq R\Gamma_{X_{\R}\times S}(\sho_{X\times \bar{X}\times S})[2d_X]$$ where as usual we regard $X\times\bar{X}$ as a complexification of $X_{\R}$.
Hence $$\Rhom_{\DXS}(\shm, B_{X\times S/S})\simeq R\Gamma_{X_{\R}\times S}(\Rhom_{\DXS}(\shm, \sho_{X\times\bar{X}\times S}))[2d_X]$$ 
Let $q_1$ denote the projection $X\times \bar{X}\times S\to X\times S$ and let $q_2$ denote the projection $X\times \bar{X}\times S\to \bar{X}\times S$.
We have 
$$R\Gamma_{X_{\R}\times S}\Rhom_{\DXS}(\shm, \sho_{X\times\ov{X}\times S})[2d_X]$$ 
$$\underset{(a')}{\simeq} R\Gamma_{X_{\R}\times S}(q_1^{-1}\Sol \shm {\otimes}_{\pOS}q_2^{-1}\sho_{\bar{X}\times S})[2d_X]$$ $$\simeq R\Gamma_{X_{\R}\times S}(\Rhom_{\pOS}(q_1^{-1}\DR \shm, \pOS){\otimes}_{\pOS}q_2^{-1}\sho_{\bar{X}\times S})[2d_X]$$ $$\underset {(b')}{\simeq}R\Gamma_{X_{\R}\times S}\Rhom_{\pOS}(q_1^{-1}\DR \shm, q_2^{-1}\sho_{\bar{X}\times S})[2d_X] $$ $$\underset{(c')}{\simeq} \Rhom_{\pOS}(\DR \shm, \sho_{\bar{X}\times S})$$ $$\underset{d'}{\simeq }\Rhom_{\pOS}(\Sol \Conj_{X,\ov X}^S(\shm), \sho_{\bar{X}\times S})$$ $$\underset{(e')}{\simeq}\Conj_{X,\ov X}^S(\shm)^{\infty}$$ where $(a')$ follows by Lemma \ref{L4}, $(b')$ follows by Lemma \ref{L2}, $(c')$ follows by a similar argument as in $(\ref{E20})$, $(d')$ follows by $(c)$ and $(e')$ follows by Theorem  \ref{T1}.
\end{proof}

\end{document}